\documentclass{amsart}

\usepackage{amsmath,amssymb,amsthm,mathtools, color, amscd}
\usepackage{hyperref}

\usepackage{array, booktabs}

\newcommand{\cJ}{\mathcal{J}}
\newcommand{\cF}{\mathcal{F}}

\newcommand{\cZ}{\mathcal{Z}}

\newcommand{\RP}{\mathbb{R}P}

\newcommand{\Q}{\mathbb{Q}}

\newcommand{\R}{\mathbb{R}}
\newcommand{\Z}{\mathbb{Z}}

\newcommand{\Tor}{\operatorname{Tor}}
\newcommand{\Hom}{\operatorname{Hom}}
\newcommand{\supp}{\operatorname{supp}}
\DeclareMathOperator{\crit}{crit}
\DeclareMathOperator{\odd}{odd}
\DeclareMathOperator{\row}{row}
\DeclareMathOperator{\im}{im}
\DeclareMathOperator{\Sq}{Sq}
\DeclareMathOperator{\Bier}{Bier}

\theoremstyle{plain}
\newtheorem{theorem}{Theorem}[section]
\newtheorem{proposition}[theorem]{Proposition}
\newtheorem{corollary}[theorem]{Corollary}
\newtheorem{lemma}[theorem]{Lemma}
\numberwithin{equation}{section}

\theoremstyle{definition}

\newtheorem{example}[theorem]{Example}

\theoremstyle{remark}
\newtheorem{remark}[theorem]{Remark}

\begin{document}
\title{Small covers as pullbacks from the simplex}

\author[S. Choi]{Suyoung Choi}
\address{Department of mathematics, Ajou University, 206, World cup-ro, Yeongtong-gu, Suwon 16499, Republic of Korea}
\email{schoi@ajou.ac.kr}

\author[H. Jang]{Hyeontae Jang}
\address{School of Mathematics, Korea Institute for Advanced Study, Seoul, South Korea}
\email{hjang1112@kias.re.kr}

\author[Y. Yoon]{Younghan Yoon}
\address{Research Institute of Basic Sciences, Ajou University, 206, World cup-ro, Yeongtong-gu, Suwon 16499, Republic of Korea}
\email{younghan300@ajou.ac.kr}

\date{\today}
\thanks{This work was supported by the National Research Foundation of Korea Grant funded by
the Korean Government (RS-2025-00521982), by a KIAS Individual Grant (MG105401) at Korea Institute for Advanced Study, and by Basic Science Research Program through the  National Research Foundation of Korea(NRF) funded by the Ministry of Education(RS-2021-NR060141)}

\begin{abstract}
We introduce and study small covers that are pullbacks from the simplex, extending pullbacks from the linear model.
Our main result gives several equivalent characterizations of this class, including torsion-freeness of odd-degree integral cohomology, vanishing of the first Steenrod square on even-degree mod~$2$ cohomology, and relations among integral and mod~$2$ Betti numbers.
\end{abstract}

\maketitle

\section{Introduction}

\emph{Small covers}, introduced in~\cite{Davis-Januszkiewicz1991}, form a class of smooth manifolds equipped with a locally standard $\Z_2^n$-action whose orbit spaces are simple convex polytopes.
They can be viewed as topological generalizations of the real loci of smooth projective toric varieties.

Let $P$ be an $n$-dimensional simple convex polytope with facets $\cF(P)=\{F_1,\dots,F_m\}$.
A (mod $2$) \emph{characteristic map} on $P$ is a map
$$
    \lambda \colon \cF(P)\to \Z_2^n
$$
such that for any vertex $v=F_{i_1}\cap \cdots \cap F_{i_n}$, the vectors $\lambda(F_{i_1}),\dots,\lambda(F_{i_n})$ form a basis of $\Z_2^n$.
For a face $F=F_{i_1}\cap \cdots \cap F_{i_k}$, let $G_F\subset \Z_2^n$ be the subgroup generated by $\lambda(F_{i_1}),\dots,\lambda(F_{i_k})$.
The small cover $M$ over $P$ determined by $\lambda$ is defined by
$$
    M:= (\Z_2^n\times P)/\sim,
$$
where $(g,p)\sim (h,q)$ if $p=q$ and $g^{-1}h\in G_{F(p)}$, and $F(p)$ denotes the unique face whose relative interior contains $p$.
See~\cite{Davis-Januszkiewicz1991,Buchstaber-Panov2015book} for more details.

It is known~\cite{Davis-Januszkiewicz1991} that the mod $2$ cohomology ring of a small cover is completely determined by the underlying polytope and the characteristic map.
The rational cohomology groups were computed in~\cite{Suciu-Trevisan2012,Choi-Park2017Cohomology}, and the integral cohomology groups were computed by~\cite{Cai-Choi2021}.
Furthermore, the ring structure of the cohomology with coefficients in a ring in which $2$ is invertible was described in~\cite{Choi-Park2020}.

A small cover as a \emph{pullback from the linear model}, introduced in~\cite{Davis-Januszkiewicz1991}, has been further studied in several later works in toric topology and related areas~\cite{Izmestiev2001,Joswig2002,Nakayama-Nishimura2005, Cai-Choi-Park2020}.
It was shown in~\cite{Davis-Januszkiewicz1991} that this condition is sufficient for torsion-freeness of integral cohomology, and was later shown in~\cite{Cai-Choi-Park2020} to be also necessary. 
In fact, it is equivalent to the vanishing of the first Steenrod square on the first mod~$2$ cohomology~\cite{Cai-Choi-Park2020}.
Earlier work~\cite{Bahri-Bendersky2000} also considered Steenrod square operations in the mod~$2$ cohomology of small covers.

In this paper, we study an analogous necessary and sufficient condition for a small cover to be a \emph{pullback from the simplex}, defined as follows.
Let $\Delta^n$ be the $n$-simplex with facet set $\cF(\Delta^n)=\{F_1,\dots,F_{n+1}\}$.
We denote by
\[
\lambda_{\Delta}\colon \cF(\Delta^n)\to \Z_2^n
\]
the characteristic map given by
$$
    \lambda_{\Delta}(F_i)=e_i \quad (1\le i\le n), \quad \text{and} \quad \lambda_{\Delta}(F_{n+1})=e_1+\cdots+e_n,
$$
up to basis change in $\Z_2^n$.
The small cover over $\Delta^n$ determined by $\lambda_{\Delta}$ is the real projective space $\RP^n$.

Let $P$ be an $n$-dimensional simple convex polytope with facet set $\cF(P)$, and let $M$ be the small cover over $P$ determined by $\lambda$.
We say that $M$ is a pullback from the simplex if there exists a map
\[
    c\colon \cF(P)\longrightarrow \cF(\Delta^n)
\]
that induces a nondegenerate simplicial map between the associated simplicial complexes, and satisfies $\lambda=\lambda_{\Delta}\circ c$.
Equivalently, $M$ is obtained from the canonical small cover $\RP^n\to \Delta^n$ by pullback along the corresponding map to the simplex.
More precisely, if $\pi\colon M\to P$ and $\pi_{\Delta}\colon \RP^n\to \Delta^n$ are the orbit maps, then the above condition means that $M$ fits into the pullback diagram
\[
\begin{CD}
M @>>> \RP^n \\
@V{\pi}VV @VV{\pi_{\Delta}}V \\
P @>{\varphi}>> \Delta^n,
\end{CD}
\]
where $\varphi\colon P\to \Delta^n$ denotes the map arising from the corresponding nondegenerate simplicial map.

\begin{theorem}\label{thm:main}
    Let~$M$ be a small cover.
    Then the following are equivalent.
    \begin{enumerate}
        \item\label{main1} $M$ is a pullback from the simplex;
        \item\label{main4} $H^{\odd}(M;\Z)$ is torsion-free;
        \item\label{main5} $H^3(M;\Z)$ is torsion-free;
        \item\label{main2} $\Sq^1$ vanishes on $H^{2k}(M;\Z_2)$ for all $k\geq 0$;
        \item\label{main3} $\Sq^1$ vanishes on $H^2(M;\Z_2)$;
        \item\label{main6} $b^{2k}(M)-b^{2k-1}(M) = b^{2k}_{\Z_2}(M)-b^{2k-1}_{\Z_2}(M)$ for every $k\ge 1$; and
        \item\label{main7} $b^2(M)-b^1(M) =b^2_{\Z_2}(M)-b^1_{\Z_2}(M)$.
    \end{enumerate}
\end{theorem}
Although we formulate Theorem~\ref{thm:main} only for small covers over simple polytopes, the same statement extends to the more general setting of shellable simplicial spheres, as will be explained in Section~\ref{sec:pre}.

In Section~\ref{sec:Steenrod}, we prove Theorem~\ref{thm:Sq1_H2_characterizes_simplex_pullback}, which establishes the equivalence of \eqref{main1}, \eqref{main2}, and \eqref{main3} of Theorem~\ref{thm:main}.
As a consequence, we obtain Corollary~\ref{cor:main}, which recovers the result in~\cite{Cai-Choi-Park2020} on pullbacks from the linear model.
\begin{corollary}\label{cor:main}
    Let~$M$ be a small cover.
    Then the following are equivalent.
    \begin{enumerate}
        \item $M$ is a pullback from the linear model;
        \item $H^\ast(M;\Z)$ is torsion-free; and
        \item $\Sq^1$ vanishes on $H^1(M;\Z_2)$.
    \end{enumerate}
\end{corollary}
In Section~\ref{sec:torsionfree}, we prove Theorem~\ref{thm:odd-torsionfree-betti-characterization}, showing that \eqref{main1}, \eqref{main4}, \eqref{main5}, \eqref{main6}, and \eqref{main7} of Theorem~\ref{thm:main} are equivalent.
The first five conditions in Theorem~\ref{thm:main} may be viewed as a natural extension of the results in~\cite{Cai-Choi-Park2020} for pullbacks from the linear model.
Conditions~\eqref{main6} and \eqref{main7} are new, although related phenomena for special cases of pullbacks from the simplex were already established in~\cite[Corollary~7.3]{Choi-Yoon-Yu2026}.

\section{Preliminaries}\label{sec:pre}

\subsection{Real toric spaces}
Let $K$ be a simplicial complex on the vertex set $[m] = \{1,\ldots,m\}$.
The \emph{real moment-angle complex} $\R\cZ_{K}$ associated with $K$ is defined by
\[
\R\cZ_{K}
= \bigcup_{\sigma \in K}
\Bigl\{ (x_1,\ldots,x_m) \in (D^1)^m \;\Big|\; x_i \in S^0 \text{ for } i \notin \sigma \Bigr\},
\]
where $D^1 = [-1,1]$ is the $1$-dimensional disk and $S^0 = \{-1,1\}$ is its boundary.

Let $\lambda \colon [m] \to \Z_2^n$ be a map with $n \le m$. 
We associate to $\lambda$ the $n \times m$ matrix over $\Z_2$
\[
\Lambda = \big[\lambda(1)\ \cdots\ \lambda(m)\big],
\]
whose $i$th column is $\lambda(i)$. 
This matrix determines a linear map $\Z_2^m \to \Z_2^n$, which we also denote by $\lambda$.
Throughout the paper, we will use $\lambda$ interchangeably to denote the original map and the induced linear map, depending on the context.

The standard sign action of $\Z_2^m$ on $(D^1)^m$ restricts to an action of $\ker \lambda$ on $\R\cZ_{K}$.
The \emph{real toric space} $M(K,\lambda)$ corresponding to the pair $(K,\lambda)$ is defined as
$$
    M(K,\lambda) = \R\cZ_{K} / \ker \lambda.
$$

It is known~\cite{Choi-Kaji-Theriault2017} that the action of $\ker \lambda$ on $\R\cZ_{K}$ is free if and only if, for every simplex $\{i_1,\ldots,i_r\} \in K$, the vectors
\[
\lambda(i_1), \ldots, \lambda(i_r)
\]
are linearly independent in $\Z_2^n$.
When this condition holds, the map $\lambda$ is called a (mod $2$) \emph{characteristic map} over $K$, and the associated matrix~$\Lambda$ is called a (mod $2$) \emph{characteristic matrix}.
In particular, this implies that $\dim K \le n-1$, and throughout we assume $\dim K = n-1$.

Let $M$ be a real toric space determined by a simplicial complex $K$ on~$[m]$ together with a (mod $2$) characteristic map $\lambda \colon [m] \to \Z_2^n$.
If $K$ is polytopal, then~$M$ is called a \emph{small cover}~\cite{Davis-Januszkiewicz1991}, and if $K$ is star-shaped, then~$M$ is called a \emph{real topological toric manifold}~\cite{Ishida-Fukukawa-Masuda2013}.
When $K$ is polytopal, this definition is equivalent to the original one given in Introduction for small covers over the corresponding simple polytope~\cite{Davis-Januszkiewicz1991,Buchstaber-Panov2015book}.

\begin{theorem}\cite[Theorem~5.12]{Davis-Januszkiewicz1991}
Let $K$ be a simplicial sphere, and let $M = M(K,\lambda)$ be a real toric space.
The $\Z_2$-cohomology ring of $M$ is given by
$$
    H^\ast(M;\Z_2)\cong \Z_2[v_1,\dots,v_m]/(I_K+J_\lambda),
$$
where each $v_j$ has degree $1$, $I_K$ is the Stanley--Reisner ideal of $K$, and
$$
    \cJ_\lambda = 
    \left\langle
    \lambda_{11}v_1 + \cdots + \lambda_{1m}v_m,\,
    \ldots,\,
    \lambda_{n1}v_1 + \cdots + \lambda_{nm}v_m
    \right\rangle,
$$
where~$\lambda(i) = (\lambda_{1i},\ldots,\lambda_{ni})$ for all~$1 \leq i \leq m$.
\end{theorem}

We denote by $\row \Lambda$ the row space of $\Lambda$. 
Each element $\omega \in \row \Lambda$ can be identified with a subset of $[m]$ via the standard correspondence between $\Z_2^m$ and the power set of $[m]$.
With this identification, $\omega$ determines the full subcomplex~$K_\omega$ of~$K$.

\begin{theorem}\cite[Theorem~1.1]{Cai-Choi2021}\label{CaiChoi}
    Let $K$ be a shellable simplicial sphere, and let~$M = M(K,\lambda)$ be a real toric space.
    For the integral cohomology group~$H^\ast(M;\Z)$ of~$M$,
    \begin{enumerate}
      \item the free part (respectively, for an odd prime $p$, the $p^k$-torsion part) of $H^i(M;\Z)$ coincides with that of $\underset{\omega \in \row\Lambda}{\bigoplus}\widetilde{H}^{i-1}(K_\omega;\Z)$, and
      \item the $2^{k+1}$-torsion subgroup of $H^i(M;\Z)$ coincides with the $2^{k}$-torsion subgroup of~$\underset{\omega \in \row\Lambda}{\bigoplus}\widetilde{H}^{i-1}(K_\omega;\Z)$.
    \end{enumerate}
\end{theorem}

\subsection{A pullback from the simplex}
In this paper, in order to avoid cumbersome notation, we formulate all results for small covers.
In fact, none of the arguments uses the polytopality of the underlying simplicial sphere.
Therefore, most of the results remain valid for real toric spaces over a simplicial sphere~$K$, while Theorem~\ref{thm:odd-torsionfree-betti-characterization} requires the additional assumption that $K$ is shellable.

Let $K$ and $L$ be simplicial complexes dual to $n$-dimensional simple convex polytopes.
A small cover~$M(K,\lambda)$ is called a \emph{pullback} from~$(L,\lambda')$ if there exists a nondegenerate simplicial map~$c \colon K \to L$ such that
\[
\lambda=\lambda'\circ c,
\]
up to basis change in $\Z_2^n$.

Let $\Delta$ denote the $(m-1)$-simplex on the vertex set $[m]$, and let $\partial\Delta$ be its boundary complex.
Define a characteristic map $\lambda_{\Delta} \colon [m] \to \Z_2^m$ by
\[
\lambda_{\Delta}(i) = e_i,
\]
where $e_i$ is the $i$th standard basis vector of $\Z_2^m$.
Similarly, define a characteristic map
\[
\lambda_{\partial\Delta} \colon [m] \to \Z_2^{m-1}
\]
by
\[
\lambda_{\partial\Delta}(i)=e_i \quad (1\le i\le m-1), \quad \text{and} \quad \lambda_{\partial\Delta}(m)=e_1+\cdots+e_{m-1}.
\]
A pullback from $(\Delta,\lambda_{\Delta})$ is called a \emph{pullback from the linear model}, and a pullback from $(\partial\Delta,\lambda_{\partial\Delta})$ is called a \emph{pullback from the simplex}.

\begin{proposition}\label{prop:pullbackiff}
Let~$M$ be an $n$-dimensional small cover with characteristic map~$\lambda$.
  \begin{enumerate}
    \item\label{prop:pullbackiff1} $M$ is a pullback from the linear model if and only if, up to basis change in~$\Z_2^n$, $\im(\lambda)=\{e_1,\ldots,e_n\}$.
\item\label{prop:pullbackiff2} $M$ is a pullback from the simplex if and only if, up to basis change in~$\Z_2^n$,
    $$
    \im(\lambda) \subseteq \{e_1,\ldots,e_n,e_1+\cdots+e_n\}.
    $$
  \end{enumerate}
\end{proposition}

\begin{remark}
The formulas for the characteristic classes of $M$ may also be viewed as pullbacks of the corresponding formulas for $\RP^n$. 
Indeed, the pullback structure gives a natural map $M\longrightarrow \RP^n$.
Let $\tau \in H^1(M;\Z_2)$ be the pullback of the standard generator $u\in H^1(\RP^n;\Z_2)$.
Then the identities
$$
    w(\RP^n)=(1+u)^{n+1} \quad \text{ and } \quad \operatorname{Wu}_i(\RP^n)=\binom{n-i}{i}u^i
$$
pull back to
$$
    w(M)=(1+\tau)^{n+1} \quad \text{ and } \quad \operatorname{Wu}_i(M)=\binom{n-i}{i}\tau^i,
$$
respectively.
\end{remark}

\begin{example}\label{ex:Subsec3dim} 
    Let $M(K,\lambda)$ be a $3$-dimensional small cover.
    It is known~\cite{Nakayama-Nishimura2005} that a $3$-dimensional small cover $M(K,\lambda)$ is orientable if and only if,
    after a basis change in $\Z_2^3$,
    $$
        \im(\lambda) \subseteq \{e_1,e_2,e_3,e_1+e_2+e_3\}.
    $$
    By Proposition~\ref{prop:pullbackiff} \eqref{prop:pullbackiff2}, it follows that $3$-dimensional orientable small cover~$M(K,\lambda)$ is always a pullback from the simplex.  
\end{example}

\begin{example}\label{ex:SubsecBier} 
    Let $K$ be a simplicial complex on a finite set~$S$ such that $K\neq 2^S$.
    For each subset $I\subset S$, define
    $$
        \bar I=\{\bar i\mid i\in I\}.
    $$
    The \emph{Bier sphere}, introduced in~\cite{Bier1992}, of $K$ is the simplicial complex on the disjoint union~$S\sqcup\overline{S}$ defined as
    $$
        \Bier(K)=\{\sigma\cup\bar{\tau}\mid \sigma\in K,\ S \setminus \tau \notin K,\ \sigma\cap\tau=\emptyset\}.
    $$
    Every Bier sphere is a PL-sphere of dimension~$|S|-2$, and it was proved in~\cite{Bjorner-Paffenholz-Sjostrand-Ziegler2005} that every Bier sphere is shellable, but in general not polytopal.
    Moreover, a Bier sphere admits a canonical characteristic map~$\lambda_{\Bier(K)}$, introduced in~\cite{Ivan_Sergeev2024}, defined as follows.
    For~$S = \{s_1,\ldots,s_{\ell}\}$, define
    $$
        \lambda_{\Bier(K)}(s_i) = \lambda_{\Bier(K)}(\overline{s_i})=
        \begin{cases}
            e_i, & \mbox{if } 1 \leq i \leq \ell-1 \\
            e_1+ \cdots + e_{\ell-1}, & i = \ell,
        \end{cases}
    $$
    where~$e_i$ is the $i$th standard basis vector of~$\Z^{\ell-1}$.
    In fact, the characteristic map~$\lambda_{\Bier(K)}$ realizes $\Bier(K)$ as the underlying simplicial complex of a smooth complete fan~\cite{Ivan-Marinko-Rade2025}.
    By Proposition~\ref{prop:pullbackiff}\eqref{prop:pullbackiff2}, a real toric manifold corresponding to the pair~$(\Bier(K),\lambda_{\Bier(K)})$ is a pullback from the simplex.
\end{example}

\section{The first Steenrod square}\label{sec:Steenrod}

Let $X$ be a topological space.
Recall that the Steenrod squares are cohomology operations
\[
\Sq^i \colon H^q(X;\Z_2) \longrightarrow H^{q+i}(X;\Z_2), \qquad i \ge 0,
\]
and the total Steenrod square is defined by
\[
\Sq = \sum_{i \ge 0} \Sq^i.
\]
They satisfy the following properties.
\begin{itemize}
  \item $\Sq^0$ is the identity,
  \item $\Sq^i(x) = 0$ for $i > \deg x$,
  \item $\Sq^{\deg x}(x) = x^2$, and
  \item $\Sq(xy) = \Sq(x)\Sq(y)$.
\end{itemize}
In particular, the first Steenrod square
\[
\Sq^1\colon H^q(X;\Z_2)\longrightarrow H^{q+1}(X;\Z_2)
\]
is given by the mod $2$ reduction of the Bockstein homomorphism associated to the
short exact sequence
\begin{equation} \label{eq:ses}
    0\longrightarrow \Z \xrightarrow{\times 2} \Z \longrightarrow \Z_2 \longrightarrow 0.  
\end{equation}
In other words,
\[
\Sq^1=\rho\circ\beta,
\]
where
\[
\beta\colon H^q(X;\Z_2)\longrightarrow H^{q+1}(X;\Z)
\]
is the Bockstein homomorphism and
\[
\rho\colon H^{q+1}(X;\Z)\longrightarrow H^{q+1}(X;\Z_2)
\]
is the mod reduction $2$ map.

Let $M = M(K,\lambda)$ be a small cover that is a pullback from the simplex via a nondegenerate simplicial map $c \colon K\longrightarrow \partial\Delta^n$.
Let $v_1,\ldots,v_m \in H^1(M;\Z_2)$ be the standard generators.
For each~$i\in [n+1] = \{1,\ldots,n+1\}$, define
\[
\tau_i:=\sum_{c(j) = i} v_j \in H^1(M;\Z_2).
\]

\begin{lemma}\label{lem:tau_equal_RT}
    One has
    $$
        \tau_1=\tau_2=\cdots=\tau_{n+1}.
    $$
\end{lemma}

\begin{proof}
    For $\ell =1, \ldots, n$, the $\ell$th relation is 
    $$
        \sum_{c(j)=\ell}v_j+\sum_{c(j)=n+1}v_j=0,
    $$
    namely, $\tau_\ell+\tau_{n+1}=0 \in H^1(M;\Z_2)$.
\end{proof}

Set $\tau = \tau_1 = \cdots = \tau_{n+1} \in H^1(M;\Z_2)$.

\begin{lemma}\label{lem:square_generator_RT}    
    For every generator $v_j\in H^1(M;\Z_2)$,
    $$
    v_j^2=\tau v_j.
    $$
    In particular, for a homogeneous element $x \in  H^q(M;\Z_2)$, we have
    $$
        \Sq(x)=(1+ \tau)^q x.
    $$
\end{lemma}

\begin{proof}
    Suppose that $c(j)=i$. 
    Then, by Lemma~\ref{lem:tau_equal_RT},
    $$
        \tau v_j= \tau_i v_j=\left(\sum_{c(k)=i} v_k\right)v_j.
    $$
    Let $c(k)=c(j)=i$ for some $k\neq j$.
    Since $c$ is nondegenerate, $\{j,k\}$ is not an edge of $K$. 
    Hence $v_kv_j\in I_K$, and so $v_kv_j=0$ in $H^\ast(M;\Z_2)$.
    
    Therefore, for each degree-one generator $v_j$, we have
    $$
        \Sq(v_j)=v_j+\Sq^1(v_j)=v_j+v_j^2 = (1+\tau)v_j.
    $$
    
    If $x=v_{i_1}\cdots v_{i_q}$ is a monomial of degree $q$, then by the multiplicativity of the total Steenrod square,
    $$
        \Sq(x)=\prod_{s=1}^q \Sq(v_{i_s}) =\prod_{s=1}^q (1+\tau)v_{i_s} =(1+\tau)^q x.
    $$
    By linearity, the same formula holds for every homogeneous $x\in H^\ast(M;\Z_2)$.
\end{proof}

For a facet $\sigma=\{u_1,\dots,u_n\}$ and for each $i \in [n] = \{1,\ldots,n\}$, let $p_i(\sigma)$ be the unique vertex such that
$$
    \bigl(\sigma\setminus\{u_i\}\bigr)\cup\{p_i(\sigma)\}
$$
is a facet. 
Since $\{\lambda(u_1),\dots,\lambda(u_n)\}$ is a basis of $\Z_2^n$,
there exists a unique subset~$S_i^\sigma \subset [n]$ such that
\[
\lambda\bigl(p_i(\sigma)\bigr)
= \sum_{j \in S_i^\sigma} \lambda(u_j).
\]
From Proposition~\ref{prop:pullbackiff} \eqref{prop:pullbackiff2}, one can see that $M$ is a pullback from the simplex if and only if, for every facet $\sigma$ and every $i\in[n]$, one has
$$
    S_i^\sigma=\{i\} \quad \text{ or } \quad S_i^\sigma=[n].
$$
In particular, by Proposition~\ref{prop:pullbackiff} \eqref{prop:pullbackiff1}, if $S_i^\sigma = \{i\}$ for all $\sigma$ and $i \in [n]$, then $M$ is a pullback from the linear model.

\begin{theorem}\label{thm:Sq1_H2_characterizes_simplex_pullback}
    Let~$M$ be a small cover.
    Then the following are equivalent.
    \begin{enumerate}
        \item\label{eq:FT1} $M$ is a pullback from the simplex;
        \item\label{eq:FT2} $\Sq^1$ vanishes on $H^{2k}(M;\Z_2)$ for all $k\geq 0$; and
        \item\label{eq:FT3} $\Sq^1$ vanishes on $H^2(M;\Z_2)$.
    \end{enumerate}
\end{theorem}

\begin{proof}
    The implication \eqref{eq:FT1} $\Rightarrow$ \eqref{eq:FT2} follows from Lemma~\ref{lem:square_generator_RT}, since for $x \in H^q(M;\Z_2)$
    $$
    \Sq^1(x) = \binom{q}{1}\tau x,
    $$
    which vanishes whenever $q$ is even.
    The implication \eqref{eq:FT2}$\Rightarrow$\eqref{eq:FT3} is obvious.
    
    Let us show the implication \eqref{eq:FT3} $\Rightarrow$ \eqref{eq:FT1}.
    Assume that $M(K,\lambda)$ is not a pullback from the simplex. 
    Then, there exist a facet $\sigma=\{u_1,\dots,u_n\}$ and an index $i\in[n]$ such that, after a change of basis of~$\Z_2^n$, we have
    $$
        \lambda(u_1)=e_1,\dots,\lambda(u_n)=e_n,
    $$ and $\lambda(p_i (\sigma))=\sum_{j\in S_i^\sigma}e_j$ for a subset $S_i^\sigma \subset[n]$ satisfying
    $$
        S_i^\sigma \neq \{i\} \quad \text{ and }\quad S_i^\sigma \neq [n].
    $$
    
    Set $p:=p_i (\sigma)$.
    Since $(\sigma\setminus\{u_i\})\cup\{p\}$ is a facet, the vector $\lambda(p)$ is not contained in the span of $\{e_j\mid j\neq i\}$, so necessarily $i\in S_i^\sigma$.
    Choose
    $$
        s\in S_i^\sigma \setminus\{i\} \quad \text{ and } \quad t \in [n]\setminus S_i^\sigma.
    $$
    We claim that
    $$
        \Sq^1(v_{u_s} v_{u_t})= v_{u_s} v_{u_t}(v_{u_s} + v_{u_t})\neq 0.
    $$
    
    For a vertex $q \in V(K)$, let $\supp_\sigma(q)\subset[n]$ be the support of $\lambda(q)$ with respect to the basis
    $$
        \lambda(u_1)=e_1,\dots,\lambda(u_n)=e_n,
    $$
    that is,
    $$
        \lambda(q)=\sum_{j\in \supp_\sigma(q)} e_j.
    $$
    
    For each $r\in[n]$, the $r$th row relation in $J_\lambda$ gives
    $$
        v_{u_r} = \sum_{\substack{q\notin \sigma\\ r\in \supp_\sigma(q)}} v_q.
    $$
    Hence
    $$
    v_{u_s}+v_{u_t} = \sum_{\substack{q\notin \sigma\\ |\supp_\sigma(q)\cap\{s,t\}|=1}} v_q.
    $$
    Since $s\in S_i^\sigma$ and $t\notin S_i^\sigma$, the vertex $p_i(\sigma)$ appears in this sum.
    
    Now let
    $$
        R':=\prod_{r\in[n]\setminus\{i,s,t\}} v_{u_r} \quad \text{ and } \quad R := R' v_{u_s} v_{u_t}.
    $$
    Then
    $$
        R' \Sq^1(v_{u_s}v_{u_t}) = R(v_{u_s}+v_{u_t}) = \sum_{\substack{q\notin \sigma\\ |\supp_\sigma(q)\cap\{s,t\}|=1}} Rv_q.
    $$
    A term $Rv_q$ is nonzero if and only if
    $$
        \bigl(\sigma\setminus\{u_i\}\bigr)\cup\{q\}
    $$
    is a facet.
    Since $K$ is a simplicial sphere, each codimension one face is contained in exactly two facets. 
    Hence the only nonzero term in the above sum is the one corresponding to $q=p$, and so
    $$
        R'  \Sq^1(v_{u_s}v_{u_t}) = R v_{p}\neq 0.
    $$
    Therefore
    $$
        \Sq^1(v_{u_s}v_{u_t})\neq 0
    $$
    as desired.
\end{proof}

\begin{corollary} \label{cor:pullback_linear_model}
    Let~$M$ be a small cover.
    Then the following are equivalent.
    \begin{enumerate}
        \item\label{eq:SS1} $M$ is a pullback from the linear model;
        \item\label{eq:SS2} $H^\ast(M;\Z)$ is torsion-free; and
        \item\label{eq:SS3} $\Sq^1$ vanishes on $H^1(M;\Z_2)$.
    \end{enumerate}
\end{corollary}
\begin{proof}
    The implication \eqref{eq:SS1}$\Rightarrow$\eqref{eq:SS2} is known by Davis--Januszkiewicz, and the implication \eqref{eq:SS2}$\Rightarrow$\eqref{eq:SS3} is obvious.
    
    Let us show the implication \eqref{eq:SS3}$\Rightarrow$\eqref{eq:SS1}.
    Since $x^2 =0$ for all $x \in H^1(M;\Z_2)$ and $H^\ast(M;\Z_2)$ is generated by degree-one elements, $\Sq^1$ vanishes also on $H^2(M;\Z_2)$.
    By Theorem~\ref{thm:Sq1_H2_characterizes_simplex_pullback}, $M$ is a pullback from the simplex.
    Suppose, toward a contradiction, that $M$ is not a pullback from the linear model.
    Then, there exists a nonzero $\tau \in H^1 (M;\Z_2)$ such that $\tau x = x^2 = 0$  for all $x \in H^1(M;\Z_2)$ by Lemma~\ref{lem:square_generator_RT}.
    It follows that $\tau z = 0$ for all $z \in H^{>0}(M;\Z_2)$.
    
    On the other hand, since $M$ is a closed $n$-manifold, Poincar\'e duality over $\Z_2$ implies that the pairing
    $$
        H^1(M;\Z_2)\times H^{n-1}(M;\Z_2)\to H^n(M;\Z_2)\cong \Z_2, \qquad (a,b)\mapsto a\smile b,
    $$
    is nondegenerate. 
    Hence, for every nonzero $\tau \in H^1(M;\Z_2)$, there exists $y\in H^{n-1}(M;\Z_2)$ such that $\tau y \neq 0$, a contradiction.
\end{proof}

Remark that Corollary~\ref{cor:pullback_linear_model} was originally proved 
in~\cite{Cai-Choi-Park2020}, and the proof above gives a new derivation.

\section{Torsion-freeness of odd-degree cohomology}\label{sec:torsionfree}

Let $X$ be a finite CW complex. For each $q\ge 0$, set
$$
    b^q(X):=\dim_{\Q} H^q(X;\Q), \quad \text{ and } \quad b^q_{\Z_2}(X):=\dim_{\Z_2} H^q(X;\Z_2).
$$
Let $T^q(X)$ denote the torsion subgroup of $H^q(X;\Z)$, so that
$$
    H^q(X;\Z)\cong \Z^{\,b^q(X)}\oplus T^q(X).
$$
We also denote by $\mu^q(X)$ the number of cyclic summands of even order in the torsion subgroup $T^q(X)$ of $H^q(X;\Z)$; equivalently,
$$
    \mu^q(X):=\dim_{\Z_2}\bigl(T^q(X)\otimes \Z_2\bigr).
$$
Note that $\mu^q(X)$ detects only torsion of even order; in particular, torsion of odd order does not contribute to $\mu^q(X)$.

The universal coefficient theorem for cohomology implies that
$$
H^q(X; \Z_2) \cong H^q(X; \Z) \otimes \Z_2 \oplus \Tor(H^{q+1}(X;\Z), \Z_2),
$$
although the decomposition is not canonical. Therefore,
\begin{align}
    b^q_{\Z_2}(X) &= \dim_{\Z_2}\bigl(H^q(X;\Z)\otimes \Z_2\bigr) + \dim_{\Z_2}\Tor(H^{q+1}(X;\Z),\Z_2) \notag\\
    &= b^q(X)+\mu^q(X)+\mu^{q+1}(X). \label{eq:bq}
\end{align}

\begin{proposition}\label{prop:betti_difference_and_odd_2_torsion}
    The following are equivalent:
    \begin{enumerate}
        \item For every $k\geq 1$,
        $$  
            b^{2k}(X)-b^{2k-1}(X) = b^{2k}_{\Z_2}(X)-b^{2k-1}_{\Z_2}(X), \text{ and }
        $$
        \item $\mu^{2k+1}(X)=0$ for every $k\ge 0$.
    \end{enumerate}
\end{proposition}

\begin{proof}
    Subtracting the identities for $q=2k$ and $q=2k-1$ in \eqref{eq:bq}, we obtain
    $$
        \bigl(b^{2k}_{\Z_2}(X)-b^{2k-1}_{\Z_2}(X)\bigr) - \bigl(b^{2k}(X)-b^{2k-1}(X)\bigr) = \mu^{2k+1}(X)-\mu^{2k-1}(X).
    $$
    Since $H^1(X;\Z)\cong \Hom(H_1(X;\Z),\Z)$ is torsion-free, we have $\mu^1(X)=0$. 
    Hence condition~\textup{(1)} is equivalent to
    $$
        \mu^{2k+1}(X)=\mu^{2k-1}(X) \qquad\text{for every }k\ge 1,
    $$
    which, together with $\mu^1(X)=0$, is equivalent to \textup{(2)}.
\end{proof}

\begin{lemma}\label{lem:betti_difference_vs_Sq1}
    Let $X$ be a finite CW complex such that $H^3(X;\Z)$ has no $\Z_{2^r}$-torsion for any $r>1$. 
    Then the following are equivalent:
    \begin{enumerate}
        \item\label{item:lemma1} $b^2(X)-b^1(X)=b^2_{\Z_2}(X)-b^1_{\Z_2}(X)$;
        \item\label{item:lemma2} The first Steenrod square vanishes on degree-two cohomology; that is,
        $$
            \Sq^1\big|_{H^{2}(X;\Z_2)}=0
        $$
    \end{enumerate}
\end{lemma}
\begin{proof}
    By Proposition~\ref{prop:betti_difference_and_odd_2_torsion}, condition~\eqref{item:lemma1} is equivalent to $\mu^{3}(X)=0$.
    Under the present assumption, this is equivalent to saying that $H^{3}(X;\Z)$ has no element of order $2$.
    We recall $Sq^1=\rho\circ\beta$.
    We claim that $\rho$ is injective on $\im(\beta)=\ker(\times 2)$, where $\times 2\colon H^{3}(X;\Z)\to H^{3}(X;\Z)$ is the map induced from the short exact sequence~\eqref{eq:ses}.     
    Indeed, let $y\in \ker(\times 2)$ and assume that $\rho(y)=0$. 
    By exactness, $y=2z$ for some $z\in H^{3}(X;\Z)$. 
    Since $2y=0$, we have $4z=0$. 
    If $y\neq 0$, then $2z\neq 0$, so $z$ has order $4$, contradicting the assumption that $H^{3}(X;\Z)$ has no element of order $2^r$ for any $r>1$. 
    Therefore $y=0$, proving the claim.
    It follows that condition~\eqref{item:lemma2} is equivalent to the vanishing of $\beta$, and hence to condition~\eqref{item:lemma1}.
\end{proof}

Let $K$ be the boundary complex of an $n$-dimensional simplicial polytope and let $\lambda \colon V(K) \to \Z_2^n$ be a characteristic map such that $\lambda = \lambda_{\partial \Delta^n} \circ c$ for a nondegenerate simplicial map $c \colon K\longrightarrow \partial\Delta^n$.
After relabeling the vertices if necessary, we may assume that $c(i)=i$ for $i=1,\dots,n$, and we regard $\lambda$ as an $n\times m$ matrix over~$\Z_2$, where $m=|V(K)|$.
For $\omega=(w_1,\dots,w_m)\in \row(\lambda)$, define
$$
    \supp(\omega):=\{\,j\in [m]\mid w_j\neq 0\,\}.
$$

Let $\rho_1,\dots,\rho_n\in \Z_2^m$ denote the rows of $\lambda$.
Then every element $\omega \in \row(\lambda)$ can be written uniquely in the form
$$
    \omega=a_1\rho_1+\cdots+a_n\rho_n
$$
for some $(a_1,\dots,a_n)\in \Z_2^n$.
Set
$$
    S_\omega:=\{ i\in [n] \mid a_i\neq 0 \},
$$
and define an even subset $\chi_\omega\subset [n+1]$ by
$$
    \chi_\omega=
        \begin{cases}
            S_\omega, & \text{if } |S_\omega| \text{ is even},\\
            S_\omega\cup\{n+1\}, & \text{if } |S_\omega| \text{ is odd}.
        \end{cases}
$$
Then $\omega\mapsto \chi_\omega$ gives a bijection between $\row(\lambda)$ and the set of all even subsets of~$[n+1]$.
Furthermore, we have
$$
    \supp(\omega)= c^{-1}(\chi_\omega) = \{j\in [m]\mid c(j)\in \chi_\omega\}.
$$
Indeed, if $c(j)\in [n]$, then $w_j=a_{c(j)}$, while if $c(j)=n+1$, then
$$
    w_j=a_1+\cdots+a_n,
$$
which is determined by the parity of $|S_\omega|$. Hence $w_j\neq 0$ if and only if $c(j)\in \chi_\omega$.

Consequently, the full subcomplex, denoted by~$K_\omega$ or~$K_{\chi_\omega}$, of~$K$ determined by~$\omega$ coincides with the full subcomplex on the vertices whose images under~$c$ lie in~$\chi_\omega$.

\begin{lemma}\label{lemma:odd-cohomology-torsion-free-simplex-pullback}
    Let $K$ be a shellable simplicial sphere, and let $\lambda$ be a characteristic map over $K$.
    Assume that $M=M(K,\lambda)$ is a pullback from the simplex.
    Then
    $$
        H^{2k+1}(M;\Z)
    $$
    is torsion-free for every $k\ge 0$.
\end{lemma}

\begin{proof}
    Let $c\colon V(K)\to [n+1]$ be the nondegenerate simplicial map corresponding to the pullback structure.
    Fix a shelling $\sigma_1,\dots,\sigma_N$ of $K$.
    For each $i$, let $r(\sigma_i)$ denote the \emph{restriction face} of $\sigma_i$, namely, the unique minimal face of $\sigma_i$ not contained in the previous term of the filtration induced by the shelling.
    Now fix an even subset $\chi\subset [n+1]$.
    For each $i$, put
    $$
        \eta_i:=\sigma_i\cap c^{-1}(\chi).
    $$
    Since $K_\chi$ is the induced subcomplex on the vertex set $c^{-1}(\chi)$, each $\eta_i$ is a simplex in $K_\chi$.
    
    Because $M(K,\lambda)$ is a pullback from the simplex, $c(\sigma_i)$ is a facet of $\partial \Delta^n$, and so it misses exactly one element for each $i$; we denote this element by $p_i$.
    Then
    $$
        |\eta_i|=
            \begin{cases}
                |\chi|-1, & \text{if } p_i\in \chi,\\[4pt]
                |\chi|, & \text{if } p_i\notin \chi.
            \end{cases}
    $$
    
    By \cite[Lemma~4.4 and Proposition~4.5]{Cai-Choi2021}, for a pure simplicial complex~$K$ with a shelling $\sigma_1, \dots, \sigma_N$ and a full subcomplex~$L$ with its vertex set $V(L) \subset V(K)$, one obtains a chain complex generated by \emph{critical simplices} associated with $L$.
    We denote its dual cochain complex by $C^\ast_{\crit}(L)$, and then its cohomology is isomorphic to $\widetilde{H}^\ast(L ; \Z)$ as graded abelian groups. 
    Moreover, by \cite[Lemma~4.6]{Cai-Choi2021}, each cochain group $C^q_{\crit}(L)$ is a free abelian group generated by (the duals of) the simplices $r(\sigma_j)$ such that $\dim(r(\sigma_j)) = q$ and $r(\sigma_j) = \sigma_j \cap V(L)$.
    
    In the case of $L=K_\chi$, this complex is freely generated by those simplices $\eta_i$ satisfying
    $$
        \eta_i=r(\sigma_i).
    $$
    
    Since every such generator is one of the simplices $\eta_i$, it follows that the cochain complex $C_{\crit}^\ast(K_\chi)$ is concentrated in the two degrees $|\chi|-2$ and $|\chi|-1$.
    Therefore we have a two-term complex of free abelian groups
    $$
        0\longrightarrow C^{|\chi|-2}_{\crit}(K_\chi) \xrightarrow{\ \delta\ } C^{|\chi|-1}_{\crit}(K_\chi) \longrightarrow 0,
    $$
    where $\delta$ is the differential of the cochain complex.
    Consequently,
    $$
        \widetilde H^{|\chi|-2}(K_\chi;\Z)=\ker(\delta)
    $$
    is torsion-free, because a subgroup of a free abelian group is free abelian.
    
    By Theorem~\ref{CaiChoi}, the group $H^{q+1}(M;\Z)$ is controlled by the groups $\widetilde H^q(K_\omega;\Z)$ for $\omega\in\row(\lambda)$.
    More precisely, odd-primary torsion is preserved, while each~$\Z_{2^r}$-summand $(r\ge 2)$ of $H^{q+1}(M;\Z)$ corresponds to a $\Z_{2^{r-1}}$-summand of some $\widetilde H^q(K_\omega;\Z)$.
    Since $K_\omega=K_{\chi_\omega}$ for every $\omega\in\row(\lambda)$, and since $\widetilde H^*(K_{\chi_\omega};\Z)$ may have torsion only in degree $|\chi_\omega|-1$, which is odd, it follows that for every odd integer $q$, the torsion subgroup of $H^q(M;\Z)$ can only be a direct sum of copies of~$\Z_2$.
    
    However, by Theorem~\ref{thm:Sq1_H2_characterizes_simplex_pullback} and a similar argument as in the proof of Lemma~\ref{lem:betti_difference_vs_Sq1}, the odd-degree cohomology of $M$ has no element of order $2$.
    Therefore $H^q(M;\Z)$ is torsion-free for every odd integer $q$, as desired.
\end{proof}

The preceding results yield the following strong characterization of pullbacks from the simplex.
\begin{theorem}\label{thm:odd-torsionfree-betti-characterization}
    Let $M$ be a small cover.
    Then the following are equivalent:
    \begin{enumerate}
        \item\label{torsion1} $M$ is a pullback from the simplex;
        \item\label{torsion2} $H^{\odd}(M;\Z)$ is torsion-free;
        \item\label{torsion3} $H^3(M;\Z)$ is torsion-free;
        \item\label{torsion4} $b^{2k}(M)-b^{2k-1}(M) = b^{2k}_{\Z_2}(M)-b^{2k-1}_{\Z_2}(M)$ for every $k\ge 1$; and
        \item\label{torsion5} $b^2(M)-b^1(M) =b^2_{\Z_2}(M)-b^1_{\Z_2}(M)$.
    \end{enumerate}
\end{theorem}

\begin{proof}
Recall that every polytopal simplicial sphere is shellable \cite{Bruggesser-Mani1971}.
Therefore, the implication \eqref{torsion1} $\Rightarrow$ \eqref{torsion2} follows from Lemma~\ref{lemma:odd-cohomology-torsion-free-simplex-pullback}, and \eqref{torsion2} $\Rightarrow$ \eqref{torsion3} is immediate.
By Proposition~\ref{prop:betti_difference_and_odd_2_torsion}, condition~\eqref{torsion4} is equivalent to the vanishing of $\mu^{2k+1}(M)$ for every $k\ge 0$.
Hence \eqref{torsion2} $\Rightarrow$ \eqref{torsion4}, and clearly \eqref{torsion4} $\Rightarrow$ \eqref{torsion5}.

Assume~\eqref{torsion3}.
Since $H^1(M;\Z)$ is torsion-free and $H^3(M;\Z)$ is torsion-free by assumption, we have $\mu^1(M)=\mu^3(M)=0$.
Hence condition~\eqref{torsion5} follows from the universal coefficient theorem.
Thus \eqref{torsion3} $\Rightarrow$ \eqref{torsion5}.

Finally, assume~\eqref{torsion5}.
Since $H^1(M;\Z)$ is torsion-free, condition~\eqref{torsion5} implies $\mu^3(M)=0$.
Therefore Lemma~\ref{lem:betti_difference_vs_Sq1} shows that
$$
    \Sq^1\colon H^2(M;\Z_2)\to H^3(M;\Z_2)
$$
vanishes, and Theorem~\ref{thm:Sq1_H2_characterizes_simplex_pullback} yields~\eqref{torsion1}.

Therefore all five conditions are equivalent.
\end{proof}

\begin{example}
    Continuing Example~\ref{ex:Subsec3dim}, let $M$ be any orientable $3$-dimensional small cover. 
    Then $b^4(M)=b^4_{\Z_2}(M)=0$ and $b^3(M)=b^3_{\Z_2}(M)=1$.
    Moreover, by Poincar\'e duality,
    $$
        b^2(M)-b^1(M)=b^2_{\Z_2}(M)-b^1_{\Z_2}(M)=0.
    $$
    Hence condition~\eqref{torsion4} in Theorem~\ref{thm:odd-torsionfree-betti-characterization} holds for~$M$.
\end{example}

\begin{example}\label{ex:Bier}
    We continue with the class of examples introduced in Example~\ref{ex:SubsecBier}, where Bier spheres give rise to real toric manifolds that are pullbacks from the simplex.

    For instance, let~$K$ be a simplicial complex on~$\{1,\ldots,9\}$ with facets
    $$
        \{1,3,8\},\quad \{1,6,7,8,9\},\quad \{2,4,5,6,8\},\quad \{2,7\},\quad \text{ and } \quad \{3,4,5,6,7,8,9\},
    $$
    and let $M=M\bigl(\Bier(K),\lambda_{\Bier(K)}\bigr)$ be the associated real toric manifold.
    For this particular~$K$, the integral and mod~$2$ Betti numbers of~$M$ are listed in Table~\ref{table:Bier}, which confirms condition~\eqref{torsion4} in Theorem~\ref{thm:odd-torsionfree-betti-characterization}.
    \begin{table}
        \centering
        \renewcommand{\arraystretch}{1.2}
        \setlength{\tabcolsep}{6pt}
        \begin{tabular}{l!{\vrule width 1pt}c!{\vrule width 0.4pt}cc!{\vrule width 0.4pt}cc!{\vrule width 0.4pt}cc!{\vrule width 0.4pt}cc!{\vrule width 0.4pt}c}
        \toprule
        $k$ & $0$ & $1$ & $2$ & $3$ & $4$ & $5$ & $6$ & $7$ & $8$ & $\ge 9$ \\
        \midrule
        $b^k(M)$
        & $1$ & $1$ & $31$ & $23$ & $43$ & $48$ & $7$ & $9$ & $0$ & $0$ \\
        $b^k_{\Z_2}(M)$
        & $1$ & $10$ & $40$ & $81$ & $101$ & $81$ & $40$ & $10$ & $1$ & $0$ \\
        \bottomrule
        \end{tabular}
        \caption{Betti numbers of a pullback from the simplex}
        \label{table:Bier}
    \end{table}

    In~\cite{Choi-Yoon-Yu2026}, the analogue of this property was established for arbitrary~$K$ by purely combinatorial methods.
    In the present framework, however, this fact acquires a topological meaning, because~$M$ is a pullback from the simplex.
    Thus, the validity of condition~\eqref{torsion4} is explained by the general topological results of this paper.
\end{example}


\end{document}